\documentclass[12pt]{amsart}
\usepackage{amssymb}
\usepackage{amsfonts}
\usepackage{color}
\usepackage{blindtext}
\usepackage{mathtools}
\usepackage{bbm}

\theoremstyle{plain}
\newtheorem{theorem}{Theorem}

\newtheorem{proposition}{Proposition}

\newtheorem{definition}{Definition}

\newcommand{\mb}{\mathbb}
\newcommand{\mc}{\mathcal}
\newcommand{\eul}{\mathfrak}

\newcommand{\bou}{_{\scriptscriptstyle{\rm b}}}

\newcommand{\A}{\eul A}
\newcommand{\Ao}{{\eul A}_{\scriptscriptstyle 0}}

\newcommand{\vp}{\varphi}

\newcommand{\Hil}{{\mc H}}
\newcommand{\C}{\mathrm{C}^{\ast}}

\newcommand{\mult}{\,{\scriptstyle \square}\,}
\newcommand{\D}{{\mc D}}

\def\x{\relax\ifmmode {\mbox{*}}\else*\fi}

\newcommand{\id}{\mathbbm{1}}
\newcommand{\ip}[2]{\langle{#1}|{#2}\rangle}

\newcommand{\ad}{^{\mbox{\scriptsize $\dag$}}}
\newcommand{\LDH}{{\mathcal L}\ad(\D,\Hil)}
\newcommand{\LDHpi}{{\mathcal L}\ad(\D_{\scriptscriptstyle\pi},\Hil_{\scriptscriptstyle\pi})}
\newcommand{\LDHpicl}{{\mathcal L}\ad(\widetilde{\D}_{\scriptscriptstyle\pi},\Hil_{\scriptscriptstyle\pi})}

\newcommand{\QA}{{\mathcal Q}_{\Ao}(\A)}

\newcommand{\SSA}{{\mathcal S}_{\Ao}(\A)}

\newcommand{\rep}{{\mc R}(\A,\Ao)}
\newcommand{\crep}{{\mc R}_c(\A,\Ao)}

\def\H{{\mathcal H}}
\newcommand{\wmult}{\mbox{\raisebox{1pt}{$\scriptscriptstyle{
\square}$}}}

\numberwithin{equation}{section}

\numberwithin{equation}{section}

\begin{document}
\hphantom{.}\vskip 1.8cm

\title[Representable functionals and derivations]{The interplay between representable functionals and derivations on Banach quasi *-algebras}
\author{Maria Stella Adamo}
\address{Dipartimento di Matematica e Informatica, Universit\`a di Palermo, I-90123 Palermo, Italy}
\email{mariastella.adamo@community.unipa.it; msadamo@unict.it}

\begin{abstract}
This note aims to highlight the link between representable functionals and derivations on a Banach quasi *-algebra, i.e. a mathematical structure that can be seen as the completion of a normed *-algebra in the case the multiplication is only separately continuous. Representable functionals and derivations have been investigated in previous papers for their importance concerning the study of the structure properties of a Banach quasi *-algebra and applications to quantum models.
\end{abstract}

\maketitle

\section{Introduction and preliminaries}

Representable functionals constitute an important tool to investigate (locally convex) quasi *-algebras for being those linear functionals that allow a GNS-like construction (see \cite{AT,Ant1,bit_repr,Frag2,Trap1}). In particular, we are interested in the case of a Banach quasi *-algebra, i.e. a locally convex quasi *-algebra whose topology is generated by a single norm in the case of separately continuous multiplication (see \cite{Ant1,Bag1,Bag2, Bag5}). As a consequence, the multiplication is defined only for certain couples of elements. On the other hand, in the last decades derivations and their properties have been extensively studied for their use in describing physical phenomena (see, for instance, \cite{AT2,Ant2,Bag8,Bag6,Brat2}). 

In the classical context, a derivation $\delta$ is a linear mapping defined on a dense *-subalgebra $\mathcal{D}(\delta)$ of a normed *-algebra $\Ao[\|\cdot\|_{\scriptscriptstyle0}]$ for which the Leibnitz rule
$$\delta(xy)=\delta(x)y+x\delta(y),\quad\forall x,y\in\mathcal{D}(\delta)$$
holds. In the framework of a Banach quasi *-algebra, the main issue concerns the definition of a weaker form of the Leibnitz rule, suitable for the new situation. This question has been addressed to the paper \cite{AT2}, where weak *-derivations appear in the *-semisimple case. In the latter paper, conditions for a weak *-derivation to be the generator of a *-automorphisms group are given, in particular a uniformly bounded norm continuous group has a closed generator. In the case $\delta$ is a special kind of weak *-derivation, we are able to give a condition of closability depending on a certain representable functional.

Going in details, we remind some definitions about Banach quasi *-algebras and representations before focusing our attention on the properties of representable functionals, in particular we are interested in notions like fully representability and *-semisimplicity, important for the ongoing work (Section 2). Sufficiently many sesquilinear forms are needed to define weak *-derivations and study them when they are given as the infinitesimal generator of a weak *-automorphisms group. For a general weak *-derivation, representable functionals still play an important role in their study (Section 3).

We start giving some preliminary notions. For further details we refer to \cite{Ant1}.
\medskip

\begin{definition} A {\em quasi *--algebra} $(\A, \Ao)$ is a pair consisting of a vector space $\A$ and a *--algebra $\Ao$ contained in $\A$ as a subspace and such that
\begin{itemize}
\item[(i)] $\A$ carries an involution $a\mapsto a^*$ extending the involution of $\Ao$;
\item[(ii)] $\A$ is  a bimodule over $\A_0$ and the module multiplications extend the multiplication of $\Ao$. In particular, the following associative laws hold:
\begin{equation}\notag \label{eq_associativity}
(xa)y = x(ay); \ \ a(xy)= (ax)y, \; \forall \, a \in \A, \,  x,y \in \Ao;
\end{equation}
\item[(iii)] $(ax)^*=x^*a^*$, for every $a \in \A$ and $x \in \Ao$.
\end{itemize}
\end{definition}

A quasi *-algebra $(\A, \Ao)$ is \emph{unital} if there is an element $\mathbbm{1}\in \Ao$, such that $a\mathbbm{1}=a=\mathbbm{1} a$, for all $a \in \A$; $\mathbbm{1}$ is unique and called the \emph{unit} of $(\A, \Ao)$.
\medskip

We now introduce a suitable class of operators in order to represent abstract quasi *-algebras. These operators are defined on a common dense domain $\D$ with values on a Hilbert space $\Hil$ and admit an adjoint with the same property. Hence, these operators are \textit{closable}. 
\smallskip

Let $\Hil$ be a Hilbert space with inner product $\ip{\cdot}{\cdot}$ and let $\D$ be a dense linear subspace of $\Hil$. We denote by $\LDH$ the set of all closable operators $X$ in $\Hil$ such that the domain of $X$ is $\D$ and the domain of the adjoint $X^{\ast}$ contains $\D$, i.e.
$$\LDH=\left\{X:\D\to\Hil:\mathcal{D}(X)=\D, \mathcal{D}(X^{\ast})\supset\D\right\}.$$
$\LDH$ is a $\mathbb{C}-$vector space with the usual sum $X+Y$ and scalar multiplication $\lambda X$ for every $X,Y\in\LDH$ and $\lambda\in\mathbb{C}$. If we define the following involution $\dagger$ and partial multiplication $\wmult$
$$X\mapsto X^{\dagger}\equiv X^{\ast}\upharpoonright_{\D}\quad\text{and}\quad X\wmult Y=X^{\dagger\ast}Y$$
defined whenever $X$ is a left multiplier of $Y$ (or $Y$ a right multiplier of $X$), i.e. $Y\D\subset\mathcal{D}(X^{\dagger\ast})$ and $X^{\dagger}\D\subset\mathcal{D}(Y^{\ast})$, then $\LDH$ is a partial *-algebra defined in \cite{Ant1}.

In $\LDH$ several topologies can be introduced (see, \cite{Ant1}). In particular, the {\em weak topology} ${\sf t_w}$ defined by the family of seminorms
$$ p_{\xi,\eta} (X)=\ip{X\xi}{\eta},\quad \xi, \eta \in \D$$
is related to a characterization of continuity of representable functionals given in a recent work \cite{AT}.
\smallskip

Due to the structure properties of quasi *-algebras, we define a *-representation similarly to the classical case, except for the requirement on the multiplications, in the following way.

\begin{definition}
A {\em *-representation} of a quasi *-algebra $(\A,\Ao)$ is a *-homomorphism $\pi:\A\to\LDHpi$, where $\D_{\scriptscriptstyle\pi}$ is a dense subspace of the Hilbert space $\Hil_{\scriptscriptstyle\pi}$, with the following properties:
\begin{itemize}
\item[(i)] $\pi(a^{\ast})=\pi(a)^{\dagger}$ for all $a\in\A$;
\item[(ii)] if $a\in\A$ and $x\in\Ao$, then $\pi(a)$ is a left multiplier of $\pi(x)$ and $\pi(a)\wmult\pi(x)=\pi(ax)$.
\end{itemize}
\end{definition}

A *-representation $\pi$ is 
\begin{itemize}
\item {\em faithful} if $a\neq0$ implies $\pi(a)\neq0$; 
\item {\em cyclic} if $\pi(\Ao)\xi$ is dense in $\Hil_{\scriptscriptstyle\pi}$ for some $\xi\in\D_{\scriptscriptstyle\pi}$.
\end{itemize}
If $(\A,\Ao)$ has a unit $\id$, then we suppose that $\pi(\id)=I_\D$, the identity operator of $\D$.
\smallskip

The {\em closure $\widetilde{\pi}$} of a *-representation $\pi$ of the quasi *-algebra $(\A,\Ao)$ in $\LDHpi$ is defined as
$$\widetilde{\pi}:\A\to\LDHpicl,\quad a\mapsto\overline{\pi(a)}_{\upharpoonright\widetilde{\D}_{\scriptscriptstyle\pi}}$$
where $\widetilde{\D}_{\scriptscriptstyle\pi}$ is the completion of $\D_{\scriptscriptstyle\pi}$ with respect to the graph topology, i.e. the topology defined by the seminorms $\eta\in\D_{\scriptscriptstyle\pi}\mapsto\left\|\pi(a)\eta\right\|$ for every $a\in\A$.
{A *-representation $\pi$ is said to be {\em closed} if $\pi=\widetilde{\pi}$.}
\medskip

A  quasi *-algebra $(\A,\Ao)$ is called a {\em normed quasi *-algebra} if a norm
$\|\cdot\|$ is defined on $\A$ with the properties
\begin{itemize}
\item[(i)]$\|a^*\|=\|a\|, \quad \forall a \in \A$;
\item[(ii)] $\Ao$ is dense in $\A$;
\item[(iii)]for every $x \in \Ao$, the map $R_x: a \in \A \to ax \in \A$ is continuous in
$\A$.
\end{itemize}
{The continuity of the involution implies that
\begin{itemize}
\item[(iii')]for every $x \in \Ao$, the map $L_x: a \in \A \to xa \in \A$ is continuous in
$\A$.
\end{itemize}
}

\begin{definition}
If $(\A,\| \cdot \|) $ is a Banach space, we say that $(\A,\Ao)$ is a Banach quasi *-algebra.\label{def} 
\end{definition}
The norm topology of $\A$ will be denoted by $\tau_n$. 

\section{Representability of Banach quasi *-algebras}
In this section we examine some properties related to representability of Banach quasi *-algebras. Among them, an important role is played by \textit{fully representability} and \textit{*-semisimplicity}. For details, see \cite{AT,Ant1,Bag5, Frag2,Trap1}.

\begin{theorem}\cite[Theorem 3.5]{Trap1}\label{thm_repr}
Let $(\A, \Ao)$ be a quasi *-algebra with unit
$\id$ and let $\omega$ be a linear functional on $(\A, \Ao)$ that satisfies the following conditions:
\begin{itemize}
    \item[(L.1)]$\omega(x^*x) \geq 0, \quad \forall x \in \Ao$;
    \item[(L.2)]$\omega(y^*a^* x)= \overline{\omega(x^*ay)}, \quad\forall x,y \in \Ao, \forall a \in \A$;
    \item[(L.3)]$\forall a \in \A$, there exists $\gamma_a >0$ such that $$|\omega(a^*x)| \leq \gamma_a \omega(x^*x)^{1/2}, \quad \forall x\in \Ao.$$\end{itemize}
Then, there exists a triple $(\pi_\omega,\lambda_\omega,\Hil_\omega)$ such that:
\begin{itemize}
\item $\pi_\omega$ is a closed cyclic *-representation ${\pi}_\omega$ of $(\A,\Ao)$, with cyclic vector $\xi_\omega$;
\item $\lambda_\omega$ is a linear map of $\A$ into $\lambda_\omega(\Ao)=\D_{\pi_\omega}$, $\xi_\omega=\lambda_\omega(\id)$ and $\pi_\omega(a)\lambda_\omega(x)=\lambda_\omega(ax)$, for every $a\in\A$ and $x\in\Ao$;
\item $ \omega(a)=\ip{{\pi}_\omega(a)\xi_\omega}{\xi_\omega}$, for every $a\in \A.$ 
\end{itemize}
This representation is unique up to unitary equivalence.
\end{theorem}

\begin{definition} A linear functional $\omega:\A\to\mathbb{C}$ satisfying (L.1)-(L.3) in Theorem \ref{thm_repr} is called {\em representable} on the quasi *-algebra $(\A,\Ao)$.
\end{definition}
The family of representable functionals on the quasi *-algebra $(\A,\Ao)$ is denoted by $\rep$.
\medskip

As in \cite[Definition 2.1]{Trap1}, given a quasi *-algebra $(\A,\Ao)$, we denote by $\QA$ the set of all sesquilinear forms on $\A\times\A$ such that
\begin{itemize}
\item[(i)] $\varphi(a,a)\geq0$ for every $a\in\A$.
\item[(ii)] $\varphi(ax,y)=\varphi(x,a^{\ast}y)$ for every $a\in\A$ and $x,y\in\Ao$
\end{itemize}

\begin{proposition}\cite[Proposition 2.9]{AT}\label{thm_GNS} Let $(\A, \Ao)$ be a quasi *-algebra with unit $\id$ and $\omega$ a linear functional on $\A$ satisfying (L.1) and (L.2). The following statements are equivalent.
\begin{itemize}
\item[(i)] $\omega$ is representable.
\item[(ii)] There exist a *-representation $\pi$ defined on a dense domain $\D_\pi$ of a Hilbert space $\H_\pi$ and a vector $\zeta \in \D_\pi$ such that
$$ \omega(a) = \ip{\pi(a)\zeta}{\zeta}, \quad \forall a \in \A.$$
\item[(iii)] There exists a sesquilinear form $\Omega^\omega \in \QA$ such that $$\omega(a)=\Omega^\omega (a,\id), \quad \forall a \in \A.$$
\end{itemize}
\end{proposition}

To every  $\omega \in \rep$ we can associate two sesquilinear forms. The first one $\Omega^\omega$, already introduced in (iii) of Proposition \ref{thm_GNS}, can be defined through the GNS representation $\pi_\omega$, with cyclic vector $\xi_\omega$. In fact, we put
\begin{equation}\label{eq_omom}\Omega^\omega (a,b)=\ip{\pi_\omega(a)\xi_\omega}{\pi_\omega(b)\xi_\omega}, \quad a,b \in \A.\end{equation}
As we have seen $\Omega^\omega\in \QA$ and
$ \omega(a)= \Omega^\omega (a,\id),$ for every $a \in \A.$

The second sesquilinear form, which we denote by $\vp_\omega$, is defined only on $\Ao\times\Ao$ by
\begin{equation} \label{sesqu_ass} \vp_\omega(x,y)= \omega(y^*x), \quad x,y \in \Ao.\end{equation}
It is clear that $\Omega^\omega$ extends $\vp_\omega.$
It is easy to see that
\begin{itemize}
\item[(i)] $\vp_\omega(x,x) \geq 0,$ for every $x \in \Ao$.
\item[(ii)] $\vp_\omega(xy,z)=\vp_\omega (y,x^*z)$ for every $x, y, z \in \Ao$.
\end{itemize}
\medskip

If $(\A, \Ao)$ is a normed quasi *-algebra, we denote by $\crep$ the subset of $\rep$ consisting of continuous functionals.
As shown in \cite{Frag2}, if $\omega\in \crep$, then the sesquilinear form $\vp_\omega$ defined in \eqref{sesqu_ass}, is {\em closable}; that is, $\vp_\omega(x_n,x_n)\to 0$, for every sequence $\{x_n\}\subset \Ao$ such that $$\mbox{$\|x_n\|\to 0$ and $\vp_\omega(x_n-x_m, x_n-x_m)\to 0$}.$$

In this case, $\vp_\omega$  has a closed extension $\overline{\vp}_\omega$ to a dense domain $D(\overline{\vp}_\omega)$ containing $\Ao$. Thus, a natural question arises: under which conditions one gets the equality $D(\overline{\vp}_\omega)=\A$? An answer to this question will be given in Proposition \ref{prop1}.
\medskip

Consider now the set
$$\A_{\mc R}:= \bigcap_{\omega \in {\mc R}_c(\A,\Ao)}D(\overline{\vp_\omega}).$$

If ${\mc R}_c(\A,\Ao)=\{0\}$, we put $\A_{\mc R}=\A$. Note that, if for every $\omega \in {\mc R}_c(\A,\Ao)$, $\vp_\omega$ is jointly
continuous with respect to the topology $\tau_n$ defined by the norm $\|\cdot\|$, we get
$\A_{\mc R}=\A$.

\medskip Set
$$\Ao^+:=\left\{\sum_{k=1}^n x_k^* x_k, \, x_k \in \Ao,\, n \in {\mb N}\right\}.$$
Then $\Ao^+$ is a wedge in $\Ao$ and we call the elements of $\Ao^+$ \emph{positive elements of} $\Ao$.
As in \cite{Frag2}, we call \emph{positive elements of} $\A$ the members of $\overline{\Ao^+}^{\tau_n}$. We set
$\A^+:=\overline{\Ao^+}^{\tau_n}$.

\begin{definition} A family of positive linear functionals $\mc F$ on
$(\A[\tau_n], \Ao)$ is called {\em sufficient} if for every $a \in
\A^+$, $a \neq 0$, there exists $\omega \in {\mc F}$ such that $\omega
(a)>0$.
\end{definition}

\begin{definition}\cite[Definition 4.1]{Frag2}\label{fully_rep} A normed quasi $^{\ast}$-algebra
$(\A[\tau_n], \Ao)$ is called {\em fully representable} if ${\mc
R}_c(\A,\Ao)$ is sufficient and $\A_{\mc R}=\A$.
\end{definition}
 
If $(\A,\Ao)$ has a unit $\id$, the condition of sufficiency required in Definition \ref{fully_rep} joined with the following condition of positivity 
\begin{equation}\tag{P}\label{P} a\in\A\;\text{and}\;\omega_x(a)\geq0\;\text{for every}\;\omega\in\mathcal{R}_c(\A,\Ao)\;\text{and}\;x\in\Ao\;\;\Rightarrow\;\;a\geq0
\end{equation}
tells us that $\omega(a)=0$ for every $\omega\in\crep$ implies $a=0$.
\medskip

We denote by $\SSA$ the subset of $\QA$ consisting of all continuous sesquilinear forms $\Omega:\A\times\A\to\mathbb{C}$ such that
$$ |\Omega(a,b)|\leq \|a\|\|b\|, \quad \forall a, b\in \A.$$
By defining $$\|\Omega\|=\displaystyle\sup_{\|a\|=\|b\|=1}\left|\Omega(a,b)\right|,$$ one obviously has $\|\Omega\| \leq 1$, for every $\Omega \in \SSA$.
\begin{definition}\label{def1}
A normed quasi *-algebra $(\mathfrak{A}[\tau_n],\Ao)$ is called {\em *-semi\-simple} if, for every $0\neq a\in\A$, there exists $\Omega\in\mathcal{S}_{\Ao}(\A)$ such that $\Omega(a,a)>0$.
\end{definition}

As we mentioned before, if $\omega\in \crep$, then the form $\vp_\omega$ defined in \eqref{sesqu_ass} is closable. For Banach quasi *-algebras this result can be improved.
\begin{proposition}\cite[Proposition 3.6]{AT}\label{prop1}
Let $(\A, \Ao)$ be a Banach quasi *-algebra with unit $\id$,  $\omega \in \crep$ and $\vp_{\omega}$ the associated sesquilinear form on $\Ao \times \Ao$ defined as in \eqref{sesqu_ass}. Then $D(\overline{\vp}_\omega)=\A$;  hence $\overline{\vp}_\omega$ is everywhere defined and bounded.
\end{proposition}

The previous proposition can be used to show the following result.

\begin{theorem}\cite[Theorem 3.9]{AT}\label{thm_fullrep_semis} Let $(\A, \Ao)$ be a Banach quasi *-algebra with unit $\id$. The following statements are equivalent. 
\begin{itemize}
\item[(i)]$\crep$ is sufficient.
\item[(ii)]$(\A,\Ao)$ is fully representable.
\end{itemize}
If the condition of positivity $(P)$ holds, (i) and  (ii) are equivalent to the following
\begin{itemize}
\item[(iii)]$(\A,\Ao)$ is *-semisimple.
\end{itemize}
\end{theorem}

Theorem \ref{thm_fullrep_semis} shows the deep relationship between fully representability and *-semisimplicity for a Banach quasi *-algebra. Under the condition of positivity (\ref{P}), the families of sesquilienar forms involved can be identified. 

\section{Derivations and their closability}
In this section, we present an appropriate definition of derivation in the case of a *-semisimple Banach quasi *-algebra. For detailed discussion, see \cite{AT2,Ant2,HP}. In the case it is a generator of a *-automorphisms group, the derivation is closed and it owns a certain spectrum. Employing sesquilinear forms introduced in the previous section, it is possible to show a condition of closability for general derivations.
\medskip

{
Through the sesquilinear forms $\varphi\in\SSA$, we can define a new multiplication in $\A$ as in \cite{ct1}.
\smallskip

Let $(\A,\Ao)$ be a *-semisimple Banach quasi *-algebra. Let $a,b\in\A$. We say that the {\it weak} multiplication $a\wmult b$ is well-defined if there exists a (necessarily unique) $c\in\A$ such that:
$$
\varphi(bx,a^*y)=\varphi(cx,y),\;\forall\, x,y\in\Ao, \forall\,\varphi\in\SSA.
$$
In this case, we put $a\wmult b:=c$.

\begin{definition}
Let $(\A,\Ao)$ be a *-semisimple Banach quasi *-algebra. An element $a\in\A$ is called \textit{bounded} if the following equivalent conditions hold
\begin{itemize}
\item[(i)] the multiplication operators $L_a$ and $R_a$ are $\|\cdot\|-$continuous;
\item[(ii)] $R_w(a)=L_w(a)=\A$, where $R_w(a)$ (resp. $L_w(a)$) is the space of universal right (resp. left) weak multipliers of $a$. 
\end{itemize}
\end{definition}
}

Let $(\A,\Ao)$ be a *-semisimple Banach quasi *-algebra and $\theta:\A\to \A$ a linear bijection. We say that $\theta$ is a weak *-automorphism of $(\A,\Ao)$ if
\begin{itemize}
\item[(i)] $\theta(a^*)=\theta(a)^*$, for every $a \in \A$;
\item[(ii)] $\theta(a)\wmult \theta(b)$ is well defined if, and only if, $a\wmult b$ is well defined and, in this case, $$\theta(a\wmult b)=\theta(a)\wmult \theta(b).$$\end{itemize}

By the previous definition, if $\theta$ is a weak *-automorphism, then $\theta^{-1}$ is a weak *-automorphism too.
\smallskip

Let $(\A,\Ao)$ be a *-semisimple Banach quasi *-algebra. Suppose that for every fixed $t\in {\mb R}$, $\beta_t$ is a weak *-automorphism of $\A$. If
 \begin{itemize} \item[(i)]$\beta_0(a)=a,$ $\forall a\in \A$  \item[(ii)] $\beta_{t+s}(a)= \beta_t(\beta_s(a))$, $\forall a\in \A$ \end{itemize}then we say that $\beta_t$ is a {\em one-parameter group of weak *-automorphisms of $(\A,\Ao)$}.
 If $\tau$ is a topology on $\A$ and the map $t\mapsto \beta_t(a)$ is $\tau$-continuous, for every $a\in \A$, we say that $\beta_t$ is a $\tau$-continuous weak *-automorphism group.
\smallskip

The definition of the infinitesimal generator of $\beta_t$ is now quite natural. If $\beta_t$ is
$\tau$-continuous, we set

$$ \mathcal{D}(\delta_\tau)=\left\{a\in \A: \lim_{t\to 0} \frac{\beta_t(a)-a}{t} \mbox{ exists in $\A[\tau]$}\right\}$$
and
$$ \delta_\tau (a)=\tau-\lim_{t\to 0} \frac{\beta_t(a)-a}{t}, \quad a \in  \mathcal{D}(\delta_\tau).$$

If the involution $a\mapsto a^*$ is $\tau$-continuous, then $a\in \mathcal{D}(\delta_\tau)$ implies $a^*\in \mathcal{D}(\delta_\tau)$ and
$\delta(a^*)=\delta(a)^*$. 

We are now giving an appropriate definition of *-derivation, weakening the Leibnitz rule thanks to sesquilinear forms.

\begin{definition}\cite[Definition 4.5]{AT2}\label{defn_deriv}
Let $(\A,\Ao)$ be a *-semisimple Banach quasi *-algebra and $\delta$ a linear map of $\mathcal{D}(\delta)$ into $\A$, where $\mathcal{D}(\delta)$ is a partial *-algebra with respect to the weak multiplication $\wmult$. We say that $\delta$ is a {\em  weak *-derivation} of $(\A,\Ao)$ if
\begin{itemize}
\item[(i)] $\Ao\subset\mathcal{D}(\delta)$
\item[(ii)]$\delta(x^*)=\delta(x)^*, \; \forall x \in \Ao$
\item[(iii)] if $a,b\in\mathcal{D}(\delta)$ and $a\wmult b$ is well defined, then $a\wmult b\in\mathcal{D}(\delta)$ and
$$\vp(\delta(a\wmult b)x,y)= \vp(bx,\delta(a)^*y)+\vp(\delta(b)x,a^*y),$$
for all $\vp\in \SSA$, for every $x,y \in \Ao$.
\end{itemize}
\end{definition}

Parallel to the case of C*-algebras, to a \textit{uniformly bounded} norm continuous weak *-automorphisms group there corresponds a closed weak *-derivation that generates the group (see \cite{Brat2}). 

\begin{theorem}\cite[Theorem 5.1]{AT2}\label{HY1} Let $\delta:\mathcal{D}(\delta)\to\A[\|\cdot\|]$ be a weak *-derivation on a *-semisimple Banach quasi *-algebra $(\A,\Ao)$. Suppose that $\delta$ is the infinitesimal generator of a uniformly bounded, {$\tau_n$-continuous} group of weak *-automorphisms of $(\A,\Ao)$. Then $\delta$ is closed; its resolvent set $\rho(\delta)$ contains {$\mathbb{R}\setminus\{0\}$} and
\begin{equation}\label{eqn_lowbound}\|\delta(a)-\lambda a\|\geq|\lambda|\,\|a\|,\quad a\in\mathcal{D}(\delta),\lambda \in {\mb R}.\end{equation}
\end{theorem}

The converse of Theorem \ref{HY1} can be proven assuming further conditions, for instance that the domain of the weak *-derivation is made of bounded elements, which turn out to be satisfied in some interesting situations such as the weak derivative in $L^p-$spaces.

\begin{theorem}\cite[Theorem 5.3]{AT2}\label{HY2} Let $\delta:\mathcal{D}(\delta)\subset\A_{\bou}\to\A[\|\cdot\|]$ be a closed weak *-derivation on a *-semisimple Banach quasi *-algebra $(\A,\Ao)$. Suppose that $\delta$ verifies the same conditions on its spectrum of Theorem \ref{HY1} and $\Ao$ is a core for every multiplication operator $\hat{L_a}$ for $a\in\A$, i.e. $\hat{L}_a=\overline{L}_a$. Then $\delta$ is the infinitesimal generator of a uniformly bounded, {$\tau_n$-continuous} group of weak *-automorphisms of $(\A,\Ao)$.
\end{theorem}

In the sequel, we focus our attention on a special case of weak *-derivations, those for which $\D(\delta)=\Ao$. These weak *-derivations are called \textit{qu*-derivations}.
\medskip

Let $(\A,\Ao)$ be a quasi *-algebra, $\delta$ be a qu*-derivation of $(\A,\Ao)$ and $\pi$ be a *-representation of $(\A,\Ao)$. Assume that \begin{equation}\label{eqn_cond_inv}\mbox{whenever }x\in\Ao \mbox{ is such that } \pi(x)=0, \mbox{ then }\pi(\delta(x))=0.\end{equation} Under this assumption, the linear map
$$\delta_{\pi}(\pi(x)):=\pi(\delta(x)),\quad x\in\Ao$$
is well defined on $\pi(\Ao)$ with values in $\pi(\A)$ and it is easily checked that $\delta_{\pi}$ is a qu*-derivation of $\Ao$ named \textit{induced by $\pi$}.

\begin{definition}
Let $(\A,\Ao)$ be a quasi *-algebra, $\delta$ be a qu*-derivation of $(\A,\Ao)$. Furthermore, let $\pi$ be a cyclic *-representation of $(\A,\Ao)$ with cyclic vector $\xi_0$ satisfying the assumption \eqref{eqn_cond_inv}. The induced qu*-derivation $\delta_{\pi}$ is \textit{spatial} if there exists $H=H^{\dagger}\in\mathcal{L}(\mathcal{D}_{\pi},\mathcal{H}_{\pi})$ such that
$$\delta_{\pi}(\pi(x))=i[H,\pi(x)],\quad x\in\Ao.$$
\end{definition}

\begin{proposition}\label{1} Let $(\A,\Ao)$ be a Banach quasi *-algebra with unit $\id$ and let $\delta$ be a qu*-derivation of $(\A,\Ao)$. Suppose that there exists a representable and continuous functional $\omega$ with $\omega(\delta(x))=0$ for
$x\in \Ao$ and let $(\H_\omega, \pi_\omega, \lambda_\omega)$ the GNS-construction associated to $\omega$. Suppose that $\pi_\omega$ is a faithful *-representation of $(\A,\Ao)$. Then there exists an element $H = H^\dag$ of ${\mathcal L}\ad(\lambda_\omega(\Ao))$ such that
$$\pi_\omega(\delta(x)) =-i[H,\pi_\omega(x)], \quad \forall x\in \Ao$$
and $\delta$ is closable.
\end{proposition}

\begin{proof} Define $H$ on $\lambda_\omega(\Ao)$ by
$$ H\lambda_\omega(x)=i\pi_\omega(\delta(x))\xi_\omega, \quad x\in \Ao$$
where $\xi_\omega=\lambda_\omega(\id)$.
We first prove that $H$ is well defined. We have
\begin{align*}
\ip{\pi_\omega(\delta(x))\xi_\omega}{\pi_\omega(y)\xi_\omega}&= \ip{\pi_\omega(y^*)\pi_\omega(\delta(x))\xi_\omega}{\xi_\omega}\\
&=\ip{\pi_\omega(y^*\delta(x))\xi_\omega}{\xi_\omega}\\
&= \ip{\pi_\omega(\delta(y^*x)-\delta(y^*)x)\xi_\omega}{\xi_\omega}\\
&=-\ip{\pi_\omega(\delta(y^*)x)\xi_\omega}{\xi_\omega}\\
&=-\ip{\pi_\omega(x)\xi_\omega}{\pi_\omega(\delta(y))\xi_\omega}.
\end{align*}
Hence if $\lambda_\omega(x)= \pi_\omega(x)\xi_\omega=0$, it follows that $\ip{\pi_\omega(\delta(x)\xi_\omega}{\pi_\omega(y)\xi_\omega}=0$, for every $y\in \Ao$.
This in turn implies that $\pi_\omega(\delta(x))\xi_\omega=0$.

The above computation shows also that $H$ is symmetric. Indeed,

\begin{align*}\ip{H\lambda_\omega(x)}{\lambda_\omega(y)}&=i\ip{\pi_\omega(\delta(x))\xi_\omega}{\pi_\omega(y)\xi_\omega}\\ &=-i\ip{\pi_\omega(x)\xi_\omega}{\pi_\omega(\delta(y))\xi_\omega}\\
&= \ip{\lambda_\omega(x)}{H\lambda_\omega(y)}.\end{align*}
Finally, if $x\in \Ao$,
\begin{align*}\pi_\omega(\delta(x))\lambda_\omega(y)&= \pi_\omega(\delta(x))\mult \pi_\omega(y) \xi_\omega\\
&=\pi_\omega(\delta(xy))\xi_\omega -\pi_\omega(x)\mult\pi_\omega(\delta(y))\xi_\omega\\
&=-iH\pi_\omega(x)\lambda_\omega(y)+i \pi_\omega(x)H\lambda_\omega(y)\\
&=-i[H,\pi_{\omega}(x)]\lambda_\omega(y), \quad \forall y \in \Ao. \quad\qed
\end{align*}
Consider now a sequence $x_n\in\Ao$ such that $\|x_n\|\to0$ and there exists $w\in\A$ for which $\|\delta(x_n)-w\|\to0$ for $n\to\infty$. Then, for every $y,z\in\Ao$, we have
\begin{align*}
\omega(z^*w^*x)&=\ip{\pi_\omega(x)\xi_\omega}{\pi_\omega(w)\wmult\pi_\omega(z)\xi_\omega}\\
&=\lim_{n\to\infty}\ip{\pi_\omega(x)\xi_\omega}{\pi_\omega(\delta(x_n))\wmult\pi_\omega(z)\xi_\omega}\\
&=i\lim_{n\to\infty}\ip{\pi_\omega(x)\xi_\omega}{H\pi_\omega(x_n)\wmult\pi_\omega(z)\xi_\omega}\\
&-i\lim_{n\to\infty}\ip{\pi_\omega(x)\xi_\omega}{\pi_\omega(x_n)\wmult H\pi_\omega(z)\xi_\omega}\\
&=-i\lim_{n\to\infty}\Omega_\omega(x,x_n\delta(z))\to0\\
\end{align*}
using Proposition \ref{prop1}. Therefore, by the cyclicity of $\xi_\omega$ and the density of $\Ao$, we obtain $\pi_\omega(w)=0$. We conclude by the faithfulness of $\pi_\omega$.
\end{proof}

It would be of interest to study the closure of the qu*-derivation in Proposition \ref{1}. Indeed, the existence of a representable and continuous functional makes that the Banach quasi *-algebra $(\A,\Ao)$ automatically *-semisimple. If $a\in\A$ is such that $\Omega(a,a)=0$ for every $\Omega\in\SSA$, then in particular
$$\SSA\ni\Omega^x_\omega(a,a):=\Omega_\omega(ax,ax)=0,\quad\forall x\in\Ao,\Omega\in\SSA.$$
This translates into $\pi_\omega(a)\lambda_\omega(x)=0$ for every $x\in\Ao$ and so $a=0$ again by faithfulness.

\bigskip
{\bf{Acknowledgement:} }
The author is grateful to the Organizers of the International Conference on Topological Algebras and Applications 2018 for taking care of this beautiful conference and the University of Tallinn (Estonia) for its hospitality.

\end{document}